\documentclass[12]{amsart}
\usepackage{amsmath,amssymb,amsthm,color,enumerate,comment,centernot,enumitem,url,cite}
\usepackage{graphicx,relsize}
\usepackage{mathtools}
\usepackage{array}
\newcommand{\ord}{\operatorname{ord}}

\newtheorem{thm}{Theorem}[section]
\newtheorem{lem}[thm]{Lemma}
\newtheorem{conj}[thm]{Conjecture}
\newtheorem{prop}[thm]{Proposition}
\newtheorem{cor}[thm]{Corollary}

\theoremstyle{remark}

\theoremstyle{definition}
    \newtheorem{defn}[thm]{Definition}

\theoremstyle{THM}

\newcommand{\abs}[1]{\left|{#1}\right|}
\def\ds{\displaystyle}

\def\R {{\mathcal R}}
\def\Z {{\mathbb Z}}

\def\Q {{\mathbb Q}}

\def\d {{\rm det}}

\def\F {{\mathbb F}}
\def\R {{\mathcal R}}

\def\Z {{\mathbb Z}}
\def\Q {{\mathbb Q}}

\def\Gal{{\mbox {Gal} }}

\def\red#1 {\textcolor{red}{#1 }}
\def\blue#1 {\textcolor{blue}{#1 }}

\usepackage{array}
\numberwithin{equation}{section}
\def\ds{\displaystyle}
\def\Z {{\mathbb Z}}
\begin{document}

\title[Monogenic Cyclotomic Compositions]{Monogenic Cyclotomic Compositions}


\author{Joshua Harrington}
\address{Department of Mathematics, Cedar Crest College, Allentown, Pennsylvania, USA}
\email[Joshua Harrington]{Joshua.Harrington@cedarcrest.edu}

\author{Lenny Jones}
\address{Professor Emeritus, Department of Mathematics, Shippensburg University, Shippensburg, Pennsylvania 17257, USA}
\email[Lenny~Jones]{lkjone@ship.edu}

\date{\today}

\begin{abstract}
Let $m$ and $n$ be positive integers, and let $p$ be a prime. Let $T(x)=\Phi_{p^m}\left(\Phi_{2^n}(x)\right)$, where $\Phi_k(x)$ is the cyclotomic polynomial of index $k$. In this article, we prove that $T(x)$ is irreducible over $\mathbb Q$ and that
\[\left\{1,\theta,\theta^2,\ldots,\theta^{2^{n-1}p^{m-1}(p-1)-1}\right\}\] is a basis for the ring of integers of $\mathbb Q(\theta)$, where $T(\theta)=0$.
\end{abstract}


\maketitle
\section{Introduction}\label{Section:Intro}

In this article, unless stated otherwise, polynomials $f(x)\in \Z[x]$ are assumed to be monic, and when we say $f(x)$ is ``irreducible", we mean irreducible over $\Q$.
Let $K$ be an algebraic number field of degree $n$ over $\Q$. For any $\theta\in K$, we let
$\Delta(\theta):=\Delta\left(1,\theta,\theta^2,\ldots,\theta^{n-1}\right)$ denote the discriminant of $\theta$. Similarly, we let $\Delta(f)$ and $\Delta(K)$ denote the discriminants over $\Q$, respectively, of the polynomial $f(x)$ and the field $K$. If $f(x)$ is irreducible, with $f(\theta)=0$ and $K=\Q(\theta)$, then we have the well-known equation \cite{Cohen}
\begin{equation}\label{Eq:Dis-Dis}
\Delta(f)=\Delta(\theta)=\left[\Z_K:\Z[\theta]\right]^2\Delta(K),
\end{equation}
where $\Z_K$ is the ring of integers of $K$.
We say $f(x)$ is \emph{monogenic} if $f(x)$ is irreducible and $\left[\Z_K:\Z[\theta]\right]=1$, or equivalently from \eqref{Eq:Dis-Dis} that $\Delta(f)=\Delta(K)$. In this situation,
$\left\{1,\theta, \theta^2,\ldots ,\theta^{n-1} \right\}$ is a basis for $\Z_K$ referred to as a \emph{power basis}. The existence of a power basis makes calculations much easier in $\Z_K$. A classic example is the cyclotomic field $K=\Q(\zeta)$, where $\zeta$ is a primitive $n$th root of unity \cite{Washington}. In this case, $\Z_K=\Z[\zeta]$ and so $\Z_K$ has the power basis $\left\{1,\zeta,\ldots,\zeta^{\phi(n)-1}\right\}$, where $\phi(n)$ is Euler's totient function. The minimal polynomial for $\zeta$ over $\Q$ is the cyclotomic polynomial of index $n$, which we denote as $\Phi_n(x)$. That is, $\Phi_n(x)$ is monogenic of degree $\phi(n)$.

Aside from the cyclotomic polynomials, there are many polynomials that are monogenic. In fact, it is shown in \cite{BSW} that the density of the irreducible monogenic polynomials of degree $n$ is $6/\pi^2\approx .607927$. However, finding specific infinite families of monogenic polynomials can be challenging. See for example \cite{ANHam,ANHus,BMT,ESW,Gaal,Gassert,GR1,GR2,JK,S1}.

Motivated by research concerning power bases of relative extensions \cite{GP1,G1,GS1,GRS}, we are led to ask the following basic related question:
\begin{equation}\label{Q:1}
  \mbox{If $f(x)$ and $g(x)$ are monogenic, when is $f\left(g(x)\right)$ monogenic?}
\end{equation}
It is not shocking that a complete answer to \eqref{Q:1} should be difficult to achieve. What is somewhat surprising is that the answer to \eqref{Q:1} requires some effort to unravel, even in seemingly ``easy" situations. For example, although $f(x)=x^2+17$, $g(x)=x^2+5$, $f(g(x))=x^4+10x^2+42$ and $g(f(x))=x^4+34x^2+294$ are all irreducible, and both $f(x)$ and $g(x)$ are monogenic, it turns out that $f(g(x))$ is monogenic but $g(f(x))$ is \emph{not} monogenic.
In this article, we focus our investigation on compositions of certain cyclotomic polynomials. More precisely, we prove the following.
\begin{thm}\label{Thm:Main}
 Let $m$ and $n$ be positive integers, and let $p$ be a prime. Define
 \[T(x):=\Phi_{p^m}\left(\Phi_{2^n}(x)\right),\]
 where $\Phi_k(x)$ is the cyclotomic polynomial of index $k$.
 Then $T(x)$ is irreducible and monogenic.
  \end{thm}

   All computer computations in this article were done using either MAGMA, Maple or Sage.

\section{Preliminaries}\label{Section:Prelim}

We first present some standard terminology and known facts.
\begin{defn}\cite{Cohen}\label{Def:ResDisc}
Let $f(x)$ and $g(x)$ be polynomials over an integral domain $\R$ with respective leading coefficients $a$ and $b$, and respective degrees $m$ and $n$. Let $K$ be the quotient field of $\R$, and let $\overline{K}$ be an algebraic closure of $K$. Suppose that the roots of $f(x)$ and $g(x)$ in $\overline{K}$ are, respectively, $r_1,r_2,\ldots ,r_m$ and $s_1,s_2,\ldots ,s_n$. Then the \emph{resultant} of $f(x)$ and $g(x)$, denoted $R(f,g)$, is defined as
   \[R(f,g):=a^nb^m\prod_{\substack{1\le i\le m\\ 1\le j \le n} }\left(r_i-s_j\right).\] The \emph{discriminant} of $f(x)$, denoted $\Delta(f)$ is defined as
  \[\Delta(f):=\dfrac{(-1)^{m(m-1)/2}}{a}R(f,f^{\prime})=(-1)^{m(m-1)/2}a^{2m-2}\prod_{i\ne j}\left(r_i-r_j\right).\]
 \end{defn}
 \begin{prop}\cite{Cohen}\label{Prop:Res}
   Let $f(x)$ and $g(x)$ be polynomials as in Definition \ref{Def:ResDisc}.
   \begin{enumerate}
   \item $R(f,g)=(-1)^{mn}R(g,f)$,
     \item $R(f,g\cdot h)=R(f,g)R(f,h)$, where $h(x)$ is any polynomial over $\R$,
     \item $\ds R(f,g)=a^n\prod_{1\le i \le m}g\left(r_i\right)$,
     \item \label{I4:Res} Let $c\in \R$. Let $u(x)=f(x+c)$ and $v(x)=g(x+c)$. Then $\ds R(u,v)=R(f,g)$.
          \item \label{I5:Res=0} $f(x)$ and $g(x)$ have a root in common in $\overline{K}$ if and only if $R(f,g)=0$.
   \end{enumerate}
    \end{prop}
    \begin{defn}
  Let $p$ be a prime and let
  \[f(x)=a_nx^n+a_{n-1}x^{n-1}+\cdots +a_1x+a_0\in \Z[x].\] We say $f(x)$ is \emph{$p$-Eisenstein} if
  \[a_n\not \equiv 0 \pmod{p},\quad  a_{0}\equiv a_1\equiv \cdots \equiv a_{n-1}\equiv 0 \pmod{p} \quad \mbox{and} \quad a_0\not \equiv 0 \pmod{p^2}.\]
  \end{defn}

\begin{thm}\cite{Ireland-Rosen} (Eisenstien's Irreducibility Criterion)\label{Thm:Eisenstein's Criterion}
  Let $p$ be a prime and let $f(x)\in \Z[x]$ be $p$-Eisenstein. Then $f(x)$ is irreducible.
\end{thm}

\begin{thm}\label{Thm:Cyclo}
Let $n$ be a positive integer, and let $\Phi_n(x)$ denote the cyclotomic polynomial of index $n$. 
\begin{enumerate}
         \item \label{I:2} \cite{Nagell} Let $p$ be a prime. Then
          \[\Phi_n\left(x^p\right)=\left\{\begin{array}{cl} \Phi_{pn}(x) & \mbox{if $n\equiv 0 \pmod{p}$}\\
      \Phi_n(x)\Phi_{pn}(x) & \mbox{if $n\not \equiv 0 \pmod{p}$}\\
      \end{array}\right.\]
      \item \label{I:3} \cite{Washington} Let $p$ be a prime, and let $m$ be a positive integer. Then
      \[\Delta\left(\Phi_{p^m}(x)\right)=\varepsilon p^{p^{m-1}(pm-m-1)},\]
      \[\mbox{where} \quad \varepsilon=\left\{\begin{array}{rc}
        -1 & \mbox{if $p^m=4$ or $p\equiv 3\pmod{4}$}\\
        1  & \mbox{otherwise}.
      \end{array}\right.\]
      \item \label{I:0} \cite{Guerrier} Let $q$ be a prime such that $q\nmid n$. Let $\ord_n(q)$ denote the order of $q$ modulo $n$. Then $\Phi_n(x)$ factors modulo $q$ into a product of $\phi(n)/\ord_n(q)$ distinct irreducible polynomials, each of degree $\ord_n(q)$. Additionally,
  for any positive integer $m$,
         \[\Phi_{q^mn}(x)\equiv \Phi_n(x)^{\phi\left(q^m\right)} \pmod{q}.\]
         \item \label{Dresden} \cite{Dresden} Let $m$ and $n$ be integers such that $0<m<n$.
             Then
         \[R\left(\Phi_m,\Phi_n\right)=\left\{\begin{array}{cl}
           p^{\phi(m)} & \mbox{if $n/m=p^a$ for some prime $p$ and integer $a\ge 1$,}\\
           1 & \mbox{otherwise.}\\
            \end{array}\right.\]
\end{enumerate}
\end{thm}

The next lemma follows by a straightforward induction on $n$.
\begin{lem}\label{Lem:Induction}
 Let $G(x), H(x)\in \Z[x]$, and let $q$ be a prime. If $G(x)\equiv H(x) \pmod{q}$, then
  \[G(x)^{q^n}\equiv H(x)^{q^n} \pmod{q^{n+1}} \quad \mbox{for all $n\ge 1$}.\]
  \end{lem}

The following theorem gives a formula for the discriminant of the composition of two polynomials. To the best of our knowledge, this result is originally due to John Cullinan \cite{Cullinan} and does not appear in the literature. Hence, for the sake of completeness, we present a proof here.

\begin{thm}\label{Thm:Disccomp}
Let $f(x)$ and $g(x)$ be polynomials in $\mathbb{Q}[x]$, with respective leading coefficients $a$ and $b$, and respective degrees $m$ and $n$.  Then
$$\Delta(f\circ g)=(-1)^{m^2n(n-1)/2}\cdot a^{n-1}b^{m(mn-n-1)}\Delta(f)^{n}R(f\circ g,g')$$
\end{thm}
\begin{proof}
Note that the degree of $f(g(x))$ is $mn$, the leading coefficient of $f(g(x))$ is $ab^m$, the degree of $f^{\prime}(x)$ is $m-1$ and the degree of $f^{\prime}(g(x))$ is $(m-1)n$.
Then
\begin{spreadlines}{.11in}
\begin{align*}
\Delta(f\circ g)
=&\frac{(-1)^{mn(mn-1)/2}}{ab^m}\cdot R(f\circ g,(f\circ g)')\\
=&\frac{(-1)^{mn(mn-1)/2}}{ab^m}\cdot R(f\circ g,(f'\circ g)\cdot g')\\
=&\frac{(-1)^{mn(mn-1)/2}}{ab^m}\cdot R(f\circ g,f'\circ g)\cdot R(f\circ g,g') \quad \mbox{(Proposition \ref{Prop:Res})} \\
=&\frac{(-1)^{mn(mn-1)/2}}{ab^m}\cdot \left((ab^m)^{(m-1)n}\prod_{\left\{\alpha : f(g(\alpha))=0\right\}}f'(g(\alpha))\right)R(f\circ g,g')\\
=&\frac{(-1)^{mn(mn-1)/2}}{ab^m}\cdot (ab^m)^{(m-1)n}\left(\prod_{\left\{\theta : f(\theta)=0\right\}}f'(\theta)\right)^{n}R(f\circ g,g')\\
=&\frac{(-1)^{mn(mn-1)/2}}{ab^m}\cdot (ab^m)^{(m-1)n}\left(\frac{1}{a^{m-1}}R(f,f')\right)^{n}R(f\circ g,g')\\
=&\frac{(-1)^{mn(mn-1)/2}}{ab^m}\cdot \frac{(ab^m)^{(m-1)n}}{a^{(m-1)n}}R(f,f')^{n}R(f\circ g,g')\quad \mbox{(Proposition \ref{Prop:Res})}\\
=&\frac{(-1)^{mn(mn-1)/2}}{ab^m}\cdot \frac{(ab^m)^{(m-1)n}}{a^{(m-1)n}}\left(\frac{a\Delta(f)}{(-1)^{m(m-1)/2}}\right)^{n}R(f\circ g,g')\\
=&(-1)^{m^2n(n-1)/2}\cdot a^{n-1}b^{m(mn-n-1)}\Delta(f)^{n}R(f\circ g,g').\qedhere
\end{align*}
\end{spreadlines}
\end{proof}

The following special case of Theorem \ref{Thm:Disccomp} will be useful to us.
\begin{cor}\label{Cor:Disccomp}
Let $f(x)\in \mathbb{Q}[x]$, such that $a$ is the leading coefficient of $f(x)$,  and $m=\deg(f)$. If $g(x)=bx^n+c\in \Q[x]$, then
$$\Delta(f\circ g)=(-1)^{mn(n-1)(m+2n)/2} a^{n-1}b^{m(mn-1)}\Delta(f)^{n}n^{mn}f(c)^{n-1}.$$
\end{cor}

\begin{proof} Note that $\deg(f\circ g)=mn$ and $\deg(g')=n-1$. Then, using Theorem \ref{Thm:Disccomp} and Proposition \ref{Prop:Res} we have
\begin{spreadlines}{.11in}
 \begin{align*}
\Delta(f\circ g)
=&(-1)^{m^2n(n-1)/2} a^{n-1}b^{(m-1)nm-m}\Delta(f)^{n}R(f\circ g,g')\\
=&(-1)^{m^2n(n-1)/2} a^{n-1}b^{(m-1)nm-m}\Delta(f)^{n}(-1)^{mn(n-1)}R(g',f\circ g)\\
=&(-1)^{m^2n(n-1)/2} a^{n-1}b^{(m-1)nm-m}\Delta(f)^{n}(-1)^{mn(n-1)}(bn)^{mn}\prod_{i=1}^{n-1}f(g(0))\\
=&(-1)^{m^2n(n-1)/2} a^{n-1}b^{(m-1)nm-m}\Delta(f)^{n}(-1)^{mn(n-1)}(bn)^{mn}f(c)^{n-1}\\
=&(-1)^{mn(n-1)(m+2n)/2} a^{n-1}b^{m(mn-1)}\Delta(f)^{n}n^{mn}f(c)^{n-1}.\qedhere
\end{align*}
\end{spreadlines}
\end{proof}

\begin{thm}\cite{Conrad}\label{Thm:Conrad}
 Let $p$ be a prime and let $f(x)\in \Z[x]$ be a monic $p$-Eisenstien polynomial with $\deg(f)=n$. Let $K=\Q(\theta)$, where $f(\theta)=0$. If $n\not \equiv 0 \pmod{p}$, then $p^{n-1} || \Delta(K)$.
\end{thm}

\begin{thm}[Dedekind] \label{Thm:Dedekind}\cite{Cohen}
Let $K=\Q(\theta)$ be a number field, $T(x)\in \Z[x]$ the monic minimal polynomial of $\theta$, and $\Z_K$ the ring of integers of $K$. Let $p$ be a prime number and let $\overline{ * }$ denote reduction of $*$ modulo $p$ (in $\Z$, $\Z[x]$ or $\Z[\theta]$). Let
\[\overline{T}(x)=\prod_{i=1}^k\overline{t_i}(x)^{e_i}\]
be the factorization of $T(x)$ modulo $p$ in $\F_p[x]$, and set
\[g(x)=\prod_{i=1}^kt_i(x),\]
where the $t_i(x)\in \Z[x]$ are arbitrary monic lifts of the $\overline{t_i}(x)$. Let $h(x)\in \Z[x]$ be a monic lift of $\overline{T}(x)/\overline{g}(x)$ and set
\[F(x)=\dfrac{g(x)h(x)-T(x)}{p}\in \Z[x].\]
Then
\[\left[\Z_K:\Z[\theta]\right]\not \equiv 0 \pmod{p} \Longleftrightarrow \gcd\left(\overline{F},\overline{g},\overline{h}\right)=1 \mbox{ in } \F_p[x].\]
\end{thm}

\subsection{The Proof of Theorem \ref{Thm:Main}}
\begin{proof}
Note that when $n=1$, the theorem is trivially true since in that case $T(x)$ is simply a translation of $\Phi_{p^m}(x)$. So, assume that $n\ge 2$.
 We first show that $T(x)$ is irreducible by showing that $T(x)$ is $p$-Eisenstein. Since $\Phi_1(x)=x-1$ and $\Phi_{2^n}(x)=x^{2^{n-1}}+1$, we have from part \eqref{I:0} of Theorem~\ref{Thm:Cyclo} with $q=p$ and $n=1$ that
\begin{equation}\label{Eq:T mod p}
T(x)=\Phi_{p^m}\left(\Phi_{2^n}(x)\right) \equiv \Phi_1(\Phi_{2^n}(x))^{p^{m-1}(p-1)}\equiv x^{2^{n-1}p^{m-1}(p-1)} \pmod{p}.
\end{equation}
It follows that only the leading coefficient of $T(x)$ is not divisible by $p$.  Since the constant term of $T(x)$ is $T(0)=\Phi_{p^m}(\Phi_{2^n}(0))=\Phi_{p^m}(1)=p$, we conclude that $T(x)$ is $p$-Eisenstein.

Next, using Corollary \ref{Cor:Disccomp} and part \eqref{I:3} of Theorem \ref{Thm:Cyclo}, we have that
\begin{align}\label{Eq:Disc(T)}
\begin{split}
\Delta(T)&=\left(\varepsilon p^{p^{m-1}(pm-m-1)}\right)^{2^{n-1}}\left(2^{n-1}\right)^{p^{m-1}(p-1)2^{n-1}}\Phi_{p^m}(1)^{2^{n-1}-1}\\
&=\varepsilon^{2^{n-1}}2^{(n-1)2^{n-1}p^{m-1}(p-1)}p^{m2^{n-1}p^{m-1}(p-1)-1}.\\
\end{split}
\end{align}
Suppose that $T(\theta)=0$, and let $K=\Q(\theta)$.
We now use Theorem \ref{Thm:Dedekind} to show that neither 2 nor $p$ divides $\left[\Z_K:\Z[\theta]\right]$.

We address the prime $p$ first. From \eqref{Eq:T mod p}, we can let $g(x)=x$ and $g(x)h(x)=x^{2^{n-1}p^{m-1}(p-1)}$ in Theorem \ref{Thm:Dedekind}, so that
\[F(x)=\dfrac{x^{2^{n-1}p^{m-1}(p-1)}-T(x)}{p}.\] It follows that $\overline{F}(0)=-1$. Consequently, $\gcd\left(\overline{F},\overline{g}\right)=1$ in $\F_p[x]$, and $\left[\Z_K:\Z[\theta]\right]\not \equiv 0 \pmod{p}$ by Theorem \ref{Thm:Dedekind}.

If $p=2$, there is nothing else to do. So, assume that $p\ne 2$. To address the prime 2, first note that
\begin{equation*}\label{Eq:Mod2}
   T(x)\equiv \Phi_{p^m}\left((x+1)^{2^{n-1}}\right)\equiv \left(\Phi_{p^m}\left(x+1\right)\right)^{2^{n-1}} \pmod{2}.
\end{equation*}
Let $a=\phi\left(p^m\right)$ and $b=\ord_{p^m}(2)$. By part \eqref{I:0} of Theorem \ref{Thm:Cyclo} with $n=p^m$ and $q=2$, we conclude that
\begin{equation}\label{Eq:Phicong}
\Phi_{p^m}\left(x+1\right)\equiv\prod_{i=1}^{a/b}t_i(x)\pmod{2},
\end{equation}
where the $t_i(x)$ are distinct irreducible polynomials each of degree $b$. Hence,
we may take
\begin{equation}\label{Eq:g}
g(x)=\prod_{i=1}^{a/b}t_i(x)\quad \mbox{and} \quad g(x)h(x)=\left(\prod_{i=1}^{a/b}t_i(x)\right)^{2^{n-1}}
\end{equation} in Theorem \ref{Thm:Dedekind}.
Then, from \eqref{Eq:Phicong}, we can write
\begin{equation}\label{Eq:g exp}
g(x)=\Phi_{p^m}\left(x+1\right)+2r(x)=\frac{(x+1)^{p^m}-1}{(x+1)^{p^{m-1}}-1}+2r(x),
\end{equation}
for some polynomial $r(x)\in \Z[x]$. For $n\ge 2$, define
\begin{equation}\label{Eq:Fn}
F_n(x):=\dfrac{g(x)^{2^{n-1}}-\Phi_{p^m}\left(x^{2^{n-1}}+1\right)}{2} \in \Z[x].
\end{equation} By Lemma \ref{Lem:Induction}, we deduce that $g(x)^{2^n}\equiv g\left(x^2\right)^{2^{n-1}} \pmod{4}$ for $n\ge 2$ so that
\[g(x)^{2^n}-\Phi_{p^m}\left(x^{2^n}+1\right)\equiv g\left(x^2\right)^{2^{n-1}}-\Phi_{p^m}\left(\left(x^2\right)^{2^{n-1}}+1\right)  \pmod{4}.\] That is,
\[2F_{n+1}(x)\equiv 2F_n\left(x^2\right) \pmod{4} \quad \mbox{for $n\ge 2$}.\]
Hence, for all $n\ge 2$, it follows that
\[F_{n+1}(x)\equiv F_n\left(x^2\right)\equiv F_n(x)^2 \pmod{2}.\] Consequently, to show that $\gcd\left(\overline{F_n},\overline{g}\right)=1$ for all $n\ge 2$, it is enough to show that $\gcd\left(\overline{F_2},\overline{g}\right)=1$.
Let
\[A(x)=(x+1)^{p^m}-1, \quad B(x)=\left(x^{2}+1\right)^{p^{m-1}}-1,\]
 \[\quad C(x)=\left(x^2+1\right)^{p^m}-1 \quad \mbox{and} \quad D(x)=(x+1)^{p^{m-1}}-1.\] Using \eqref{Eq:g exp} and \eqref{Eq:Fn}, we have that
\begin{align}\label{Eq:F21}
F_2(x)&=\dfrac{\left(\dfrac{A(x)}{D(x)}+2r(x)\right)^{2}-\dfrac{C(x)}{B(x)}}{2} \nonumber\\
&=\dfrac{A(x)^2B(x)-C(x)D(x)^2}{2B(x)D(x)^2}+2\left(\dfrac{r(x)A(x)+r(x)^2D(x)}{D(x)}\right).
\end{align}
Expanding $A(x)^2$, we get
\begin{align*}
  A(x)^2&=(x+1)^{2p^m}-2(x+1)^{p^m}+1\\
  &=\left(x^2+1+2x\right)^{p^m}-2(x+1)^{p^m}+1\\
  &=\left(\left(x^2+1\right)^{p^m}+p^m\left(x^2+1\right)^{p^m-1}(2x)+\cdots +(2x)^{p^m}\right)-2(x+1)^{p^m}+1\\
  &=\left(x^2+1\right)^{p^m}+2p^mx\left(x^2+1\right)^{p^m-1}+4u(x)-2(x+1)^{p^m}+1,
\end{align*}
where
\[4u(x)=\binom{p^m}{2}\left(x^2+1\right)^{p^m-2}(2x)^2+\cdots +\binom{p^m}{p^m-1}\left(x^2+1\right)(2x)^{p^m-1}+(2x)^{p^m}.\]
Similarly,
\[D(x)^2=\left(x^2+1\right)^{p^{m-1}}+2p^{m-1}x\left(x^2+1\right)^{p^{m-1}-1}+4v(x)-2(x+1)^{p^{m-1}}+1,\]
where 
\[4v(x)=\binom{p^{m-1}}{2}\left(x^2+1\right)^{p^{m-1}-2}(2x)^2+\cdots +(2x)^{p^{m-1}}.\]
Substituting the expressions above for $A(x)^2$ and $D(x)^2$ into the numerator of the first term of \eqref{Eq:F21} and cancelling the factor of 2 yields
\begin{equation*}\label{Eq:F22}
  F_2(x)=\dfrac{E(x)}{B(x)D(x)^2}+2\left(\dfrac{u(x)B(x)-v(x)C(x)}{B(x)D(x)^2}+\dfrac{r(x)A(x)+r(x)^2D(x)}{D(x)}\right),
\end{equation*}
where
\begin{multline*}\label{Eq:E}
E(x)=-\left(x^2+1\right)^{p^m}+\left(x^2+1\right)^{p^{m-1}}+p^mx\left(x^2+1\right)^{p^m-1+p^{m-1}}\\
-p^mx\left(x^2+1\right)^{p^m-1}-(x+1)^{p^m}\left(x^2+1\right)^{p^{m-1}}\\
+(x+1)^{p^m}-\left(x^2+1\right)^{p^m-1+p^{m-1}}p^{m-1}x\\
+\left(x^2+1\right)^{p^m}(x+1)^{p^{m-1}}\\
+p^{m-1}x\left(x^2+1\right)^{p^{m-1}-1}-(x+1)^{p^{m-1}}.
\end{multline*}
Then, reduction of $F_2(x)$ modulo 2 produces
\begin{equation}\label{Eq:F22 mod 2}
  \overline{F_2}(x)=\dfrac{\overline{E}(x)}{\left((x+1)^{p^{m-1}}-1\right)^4},
\end{equation}
where
\begin{multline*}\label{Eq:Ebar}
\overline{E}(x)=\left(x+1\right)^{2p^m}+\left(x+1\right)^{2p^{m-1}}+x\left(x+1\right)^{2\left(p^m-1\right)}+(x+1)^{p^m}\left(x+1\right)^{2p^{m-1}}\\
+(x+1)^{p^m}+\left(x+1\right)^{2p^m}(x+1)^{p^{m-1}}
+x\left(x+1\right)^{2\left(p^{m-1}-1\right)}+(x+1)^{p^{m-1}}.
\end{multline*}
  We claim that $\gcd\left(\overline{F_2},\overline{g}\right)=1$. To establish this claim, assume to the contrary that there exists $\alpha$ in some algebraic closure of $\F_2$ such that
  \begin{equation}\label{Eq:Assumption}
  \overline{g}(\alpha)=0=\overline{F_2}(\alpha).
  \end{equation} Then, from \eqref{Eq:Phicong} and \eqref{Eq:g}, we have $\Phi_{p^m}(\alpha+1)=0$ so that $\left(\alpha+1\right)^{p^m}=1$.
    Hence, it follows from \eqref{Eq:F22 mod 2} that
\begin{spreadlines}{.11in}
\begin{align*}
 \overline{F_2}(\alpha)&= \dfrac{\alpha(\alpha+1)^{2\left(p^{m-1}-1\right)}\left(\left((\alpha+1)^{p^{m-1}}\right)^{(p-1)/2}-1\right)^4}{\left((\alpha+1)^{p^{m-1}}-1\right)^4}\\
&=\alpha(\alpha+1)^{2\left(p^{m-1}-1\right)}\left(\prod_{\substack{d\mid \frac{p-1}{2}\\d>1}}\Phi_d\left((\alpha+1)^{p^{m-1}}\right)\right)^4.
\end{align*}
\end{spreadlines}
Thus, since $\overline{F_2}(\alpha)=0$ and both $\overline{g}(0)$ and $\overline{g}(-1)$ are nonzero, we conclude that $\Phi_{p^m}(x+1)$ and $\Phi_{d}\left((x+1)^{p^{m-1}}\right)$ have the root $\alpha$ in common for some divisor $d>1$ of $(p-1)/2$. Repeated application of part \eqref{I:2} of Theorem \ref{Thm:Cyclo} yields
\begin{equation*}
 \Phi_{d}\left((x+1)^{p^{m-1}}\right)=\prod_{j=0}^{m-1}\Phi_{p^jd}(x+1).
\end{equation*}
Therefore, $\Phi_{p^m}(x+1)$ and $\Phi_{p^jd}(x+1)$ have the root $\alpha$ in common for some $j$ with $0\le j\le m-1$. However, recall that $d$ is a divisor of $(p-1)/2$ with $d>1$ and note that $p^jd<p^m$. Hence, $\gcd(d,p)=1$ and $p^m/p^jd$ is not an integer. Thus, from part \eqref{Dresden} of Theorem \ref{Thm:Cyclo}, we have that $R\left(\Phi_{p^jd},\Phi_{p^m}\right)=1$, which contradicts \eqref{Eq:Assumption} by part \eqref{I5:Res=0} of Proposition \ref{Prop:Res}. Therefore, by Theorem \ref{Thm:Dedekind}, we have established that $\left[\Z_K:\Z[\theta]\right]\not \equiv 0 \pmod{2}$, which completes the proof of the theorem.
\end{proof}

\section{Final Remarks}
We certainly do not claim that Theorem \ref{Thm:Main} represents the most general result possible, since computationally there appear to be many other situations where the composition of two cyclotomic polynomials is indeed monogenic. The main difficulty in establishing a result that is more general than Theorem \ref{Thm:Main} arises from the handling of the prime divisors of the resultant appearing in the formula given in Theorem \ref{Thm:Disccomp} for the discriminant of the composition. On the other hand, while many monogenic cyclotomic compositions exist that are not described by Theorem \ref{Thm:Main}, we should point out that slight deviations from the polynomials in Theorem \ref{Thm:Main} yield compositions that are not monogenic. For example, it is not difficult to show that $\Phi_4\left(\Phi_{3^n}(x)\right)$ and $\Phi_3\left(\Phi_{5^n}(x)\right)$ are reducible for all $n\ge 1$, and although $\Phi_2\left(\Phi_{25}(x)\right)$ is irreducible, it is not monogenic since the ring of integers does not possess a power basis. Perhaps these exceptions can be completely determined, but that is a topic for future research.

\end{document}